\documentclass[10pt]{amsart}
\usepackage{graphicx}
\usepackage[active]{srcltx}
\usepackage{amsthm,mathrsfs}
\usepackage{a4wide}
\usepackage{amsfonts}
\usepackage{amsmath}
\usepackage{amssymb}
\newtheorem{lemma}{Lemma}
\newtheorem{theorem}{Theorem}
\newtheorem{corollary}{Corollary}
\newtheorem{proposition}{Proposition}
\usepackage{float}
\usepackage{url}
\usepackage{dsfont}
\usepackage{enumitem}

\newcommand {\E} {\mathbb{E}}

\newcommand {\p} {\mathbb{P}}

\newcommand {\ve} {\varepsilon}

\makeatletter\def\blfootnote{\xdef\@thefnmark{}\@footnotetext}\makeatother
\allowdisplaybreaks
\title[Large deviation inequalities for empirical processes]{\bf Fully explicit large deviation inequalities for empirical processes with applications to information-based complexity}

\author{Christoph Aistleitner} 
\address{ Institute of Financial Mathematics and applied Number Theory,
University Linz, Altenbergerstrasse~69, 4040~Linz, Austria}
\email{aistleitner@math.tugraz.at}

\thanks{The author is supported by a Schr\"odinger scholarship of the
Austrian Research Foundation (FWF)}


\begin{document}

\begin{abstract}
In the present paper we obtain fully explicit large deviation inequalities for empirical processes indexed by a Vapnik--Chervonenkis class of sets (or functions). Furthermore we illustrate the importance of such results for the theory of information-based complexity.
\end{abstract}

\date{}
\maketitle

\section{Introduction and statement of results} \label{sect_int}

Let $\xi_1, \xi_2, \dots$ be independent random variables taking values in a measurable space $(\mathbb{X}, \mathcal{X})$. The $n$-th empirical measure $\mathbb{P}_n$ is defined by
$$
\mathbb{P}_n := \frac{1}{n} \sum_{i=1}^n \delta_{\xi_i},
$$
where $\delta_x$ denotes the Dirac measure centered at $x$. Furthermore, the $n$-th empirical process $\alpha_n$ is defined by
$$
\alpha_n := \sqrt{n} \left( \mathbb{P}_n - \frac{1}{n} \sum_{i=1}^n \mathbb{P}^{(i)} \right),
$$
where $\mathbb{P}^{(i)}$ is the distribution of $\xi_i$. In the present paper we will only be concerned with the case when all $\xi_i$'s have the same distribution $\mathbb{P}$, in which case the definition of the empirical process reduces to
$$
\alpha_n := \sqrt{n} \left( \mathbb{P}_n - \mathbb{P} \right).
$$
Let $\mathcal{C}$ denote a class of elements of $\mathcal{X}$ (that is, a class of real-valued measurable subsets of $\mathbb{X}$); then for an element $C$ of $\mathcal{C}$ we have
\begin{equation} \label{ac}
\alpha_n(C) = n^{-1/2} \sum_{i=1}^n \left( \mathds{1}_C (\xi_i) - \mathbb{P} (C) \right),
\end{equation}
where $\mathds{1}_C$ denotes the indicator function of $C$. For a measure $\mathbb{Q}$ on $(\mathbb{X}, \mathcal{X})$, we set $\mathbb{Q}(f) = \int f ~d\mathbb{Q}$. Let $\mathcal{F}$ be a class of measurable functions on $(\mathbb{X},\mathcal{X})$. Then for $f \in \mathcal{F}$ we have $\mathbb{P} (f) = \mathbb{E} (f(\xi_i))$, and thus we are led to the definition
$$
\alpha_n (f) = n^{-1/2} \sum_{i=1}^n \left( f(\xi_i) - \mathbb{E} (f(\xi_i)) \right), \qquad f \in \mathcal{F},
$$
as a straightforward generalization of~\eqref{ac}. To understand the convergence properties of $\alpha_n$ it is a natural question to investigate the quantities
\begin{equation} \label{ans}
\sup_{C \in \mathcal{C}} |\alpha_n(C)| \qquad \textrm{and} \qquad \sup_{f \in \mathcal{F}} |\alpha_n(f)|,
\end{equation}
respectively. In the simple case of $\xi_1, \xi_2, \dots$ being ``ordinary'' real random variables and $\mathcal{C}$ being the class of subintervals of $\mathbb{R}$, a result of this type is the well-known Glivenko--Cantelli theorem. In the general case it turns out that the asymptotic order of the quantities in~\eqref{ans} depends on the complexity (in an appropriate sense) of the classes $\mathcal{C}$ and $\mathcal{F}$, respectively. For a general introduction into this theory, see the book of van der Vaart and Wellner~\cite{vdv}.\\

It turns out that a particular instance in which this problem can be efficiently handled is the case of $\mathcal{C}$ being a \emph{Vapnik--Chervonenkis class} (VC class). Let $A$ be a finite subset of $\mathbb{X}$. Then $\mathcal{C}$ is said to shatter $A$ if for every subset $B$ of $A$ there exists a $C \in \mathcal{C}$ such that $B = A \cap C$. If there exists a largest finite number $d$ such that $\mathcal{C}$ shatters at least one set of cardinality $d$, then $\mathcal{C}$ is called a VC class of index $d$. To avoid measurability problems we will assume throughout this paper that $\mathcal{C}$ is countable.  The notion of VC classes is well-established in the theory of empirical processes, since there exist strong bounds for the entropy of such classes. More precisely, there is a large number of concentration inequalities or large deviation inequalities for empirical processes indexed by a class of sets (or functions), provided that this class of sets (or functions) satisfies certain entropy 
conditions (see, for example,~\cite{tala2,tala}). On the other hand, there are many results on the entropy (such as, e.g., results on $L^2$-covering numbers) of VC classes, including results of Pollard~\cite{poll} and Dudley~\cite{dud} and culminating in the work of Haussler~\cite{hauss}.\\

Large deviation inequalities for empirical processes are an important tool in information-based complexity, where they are used to establish the existence of low-discrepancy point sets (that is, point sets having an empirical distribution which is ``close'' to a desired distribution). This connection is described in more detail in Section~\ref{sec:ibc}. However, as far as I know, there exist no reasonable large deviation inequalities for empirical processes indexed by a VC class of sets (or functions) which are \emph{completely explicit}, in the sense that they contain neither unspecified constants nor quantities depending on the entropy of the class of test sets in a complicated (inexplicit) way.\footnote{The only exception which I know of is Lemma 2 in~\cite{ad}, which is based on results of Alexander~\cite{alex}; however, the involved constants there are ridiculously large, and are of no practical use in the applications in information-based complexity which we have in 
mind.} However, for applications in numerical analysis any results containing unspecified quantities are clearly of limited use, and accordingly the main point in writing this note was to provide completely explicit inequalities for this problem, with a view towards applications in information-based complexity (and with a reasonable size of the constants involved).\\

Theorem~\ref{th1} below is a large deviation inequality for the supremum of an empirical process indexed by a VC class of sets. The subsequent Theorem~\ref{th2} is a similar result for empirical processes indexed by functions.

\begin{theorem} \label{th1}
Let $\xi_1, \dots, \xi_n$ be independent random variables with values in some measurable space $(\mathbb{X}, \mathcal{X})$, and having common distribution $\mathbb{P}$. Let $\mathcal{C}$ be a countable family of elements of $\mathcal{X}$. Assume that $\mathcal{C}$ is a VC class of index $d$. Then for
$$
Z = \sup_{C \in \mathcal{C}} \left|\sum_{i=1}^n \left(\mathds{1}_C(\xi_i) - \mathbb{P}(C) \right) \right|
$$
and for all positive real numbers $\eta$ and $x$ we have
$$
\textup{Pr} \left( Z \geq 1015(1+\eta) d \left(1+ \sqrt{1+\frac{2n}{1015 d}} \right) + 2 \sqrt{n x} + \left(2.5+32\eta^{-1}\right) x \right) \leq 2 e^{-x}.
$$
\end{theorem}

From Theorem~\ref{th1}, by choosing $\eta=1/100$ and assuming that $n \geq d$, we immediately obtain the following corollary, which is even easier to apply.
\begin{corollary} \label{co1}
If in Theorem~\ref{th1} we additionally assume that $n \geq d$, then
$$
\textup{Pr} \left( Z \geq 2052 \sqrt{dn} + 2 \sqrt{n x} + 3203 x \right) \leq 2 e^{-x},
$$
for every $x>0$.
\end{corollary}

For the statement of Theorem~\ref{th2} we need the notion of a \emph{VC subgraph class}. Let $f$ be a function which maps $\mathbb{X}$ to $\mathbb{R}$. Then the \emph{subgraph} $C_f$ of $f$ is the set
$$
\big\{(x,t) \in \mathbb{X} \times \mathbb{R}:~t \leq f(x)\big\}.
$$
For a given family of functions $\mathcal{F}$, let $\mathcal{C}$ denote the class of all sets $C_f,~ f \in \mathcal{F}$. If $\mathcal{C}$ is a VC class, then $\mathcal{F}$ is called a VC subgraph class, and by definition the VC index of $\mathcal{F}$ is the VC index of $\mathcal{C}$. With this notation, we can formulate the following theorem.

\begin{theorem} \label{th2}
Let $\xi_1, \dots, \xi_n$ be independent random variables with values in some measurable space $(\mathbb{X}, \mathcal{X})$, and having common distribution $\mathbb{P}$. Let $\mathcal{F}$ be a countable family of real-valued, measurable, $\p$-centered functions on $(\mathbb{X},\mathcal{X})$ such that $\|f\|_\infty \leq 1$ for all $f \in \mathcal{F}$. Assume furthermore that $\mathcal{F}$ is a VC subgraph class of index $d$. Let 
$$
\sigma^2 = \sup_{f \in \mathcal{F}}~ \mathbb{E}(f^2(\xi_i)).
$$
Then for all positive real numbers $\eta$ and $x$, and for 
$$
Z=\sup_{f \in \mathcal{F}} \left| \sum_{i=1}^n f(\xi_i) \right|,
$$
we have
$$
\textup{Pr} \left( Z > (1 + \eta) \left( 60 \sqrt{d} \sqrt{n \sigma^2 \log (55/\sigma)} + 14400 d \log (55/\sigma) \right) + \sigma \sqrt{8 n x} + \kappa(\eta) x \right) \leq e^{-x}
$$
where $\kappa(\eta)=2.5 + 32 \eta^{-1}$. If we are willing to choose $\sigma=1$, then we obtain
$$
\textup{Pr} \left( Z > (1 + \eta) \left( 121 \sqrt{d} \sqrt{n} + 58000 d \right) + \sqrt{8 n x} + \kappa(\eta) x \right) \leq e^{-x},
$$
where $\kappa(\eta)$ is as above.
\end{theorem}

\section{Applications to information-based complexity: existence of low-discrepancy point sets} \label{sec:ibc}

The theory of information-based complexity is concerned with the amount of information on an object (e.g. a function) which is necessary to approximate a quantity depending on this object up to a given precision. More specifically, in a typical instance the ``amount of information'' means the number of function evaluations, and the problem is to approximate the integral of a $d$-variate function $f$ (which may be assumed to come from a certain class of functions, e.g. satisfying certain smoothness assumptions). A standard method for approximating such integrals is the \emph{Quasi-Monte Carlo} (QMC) method, which is based on the fact that for points $\mathbf{x}_1, \dots, \mathbf{x}_n \in [0,1]^d$ and for a $d$-variate function $f$ having bounded variation $\textup{Var}_{HK} f$ on $[0,1]^d$ we have
\begin{equation} \label{koks}
\left|\int_{[0,1]^d} f(\mathbf{x})~d\mathbf{x} - \frac{1}{n} \sum_{i=1}^n f(\mathbf{x}_i) \right| \leq \left(\textup{Var}_{HK} f \right) D_n^* (\mathbf{x}_1, \dots, \mathbf{x}_n).
\end{equation}
Here $D_n^*(\mathbf{x}_1, \dots, \mathbf{x}_n)$ is the \emph{star-discrepancy} of the point set $\mathbf{x}_1, \dots, \mathbf{x}_n$, which is defined as 
\begin{equation} \label{dn}
D_n^*(\mathbf{x}_1, \dots, \mathbf{x}_n) = \sup_{A \in \mathcal{A}} \left| \frac{1}{n} \sum_{i=1}^n \mathds{1}_A (\mathbf{x}_i) - \textup{vol}(A)\right|,
\end{equation}
where $\mathcal{A}$ denotes the class of all axis-parallel boxes in $[0,1]^d$ which have one vertex at the origin. Compare the definition of the discrepancy to~\eqref{ac}; however, note also that $\mathbf{x}_1, \dots, \mathbf{x}_n$ are assumed to be \emph{deterministic} points. The inequality~\eqref{koks}, which is called \emph{Koksma--Hlawka inequality}, essentially says that point sets having small discrepancy lead to small errors in QMC integration. Accordingly, the question asking for the necessary number of function evaluations in order to approximate the integral of $f$ can be reduced to the question asking for the existence of low-discrepancy point sets. This theory is quite well-understood in the case of moderate $d$ and ``large'' $n$, where low-discrepancy point sets can be constructed using certain ``nets'' (see, for example,~\cite{dpd}). However, the state of our knowledge is very unsatisfactory in the case when $d$ is large and $n$ is assumed to be moderate in comparison with 
$d$, a problem which is the source of much ongoing research (for more information on this problem confer the three volumes on \emph{Tractability of Multivariate Problems} by Novak and Wo{\'z}niakowski, in particular volume II~\cite{nwt} which is devoted to linear functionals such as integration).\\

Let $n^*(\ve,d)$ denote the \emph{inverse of the star-discrepancy}; that is, the smallest possible cardinality of a point set in $[0,1]^d$ having star-discrepancy at most $\ve$, for given $\ve \in (0,1)$. A fundamental theorem of Heinrich, Novak, Wasilkowski and Wo{\'z}niakowski~\cite{hnww} asserts that 
$$
n^*(\ve,d) \leq c_{\textup{abs}} d \ve^{-2}.
$$
Surprisingly, this result follows directly from a combination of a large deviation inequality of Talagrand~\cite{tala2} with the entropy bounds for VC classes of Haussler~\cite{hauss}. Similar results for general discrepancies can be obtained in exactly the same way; however, they all contain the unspecified constant $c_{\textup{abs}}$, which is inherited from Talagrand's paper~\cite{tala2}.\footnote{It seems that in principle it should be possible to obtain a completely explicit version of the main result of~\cite{tala2} by repeating Talagrand's arguments and keeping track of all the constants involved; however, it is clear that this would require a huge amount of effort. Talagrand himself writes, somewhere near the end of his 50-page paper: \emph{``Figuring out the best possible dependence [of the constants] given by this approach requires checking many computational details and more energy than the author has left at this point.''}} The main point for writing the present note was to provide a numerically 
fully explicit version of such ``inverse of the discrepancy'' results, by avoiding this particular result of Talagrand and using instead a different, fully explicit large deviation inequality together with the explicit entropy bounds of Haussler.\\

To define a general discrepancy, we assume that $(\mathbb{X},\mathcal{X})$ is a measurable space and that $\mathbb{P}$ is a probability measure on this space. Let $\mathcal{C}$ denote a family of elements of $\mathcal{X}$, and assume that $x_1, \dots, x_n$ are $n$ elements of $\mathbb{X}$. Then the discrepancy $D_n^{(\mathcal{C},\mathbb{P})}$ of the points $x_1, \dots, x_n$ is defined as 
$$
D_n^{(\mathcal{C},\mathbb{P})}(x_1, \dots, x_n) = \sup_{C \in \mathcal{C}} \left| \frac{1}{n} \sum_{i=1}^n \mathds{1}_C (x_i) - \mathbb{P} (C) \right|.
$$

As a consequence of Theorem~\ref{th1} we will obtain the following result.

\begin{theorem} \label{th3}
Let $(\mathbb{X},\mathcal{X},\mathbb{P})$ be a probability space and let $\mathcal{C}$ be a countable family of elements of $\mathcal{X}$. Then for every given $\ve \in (0,1)$ there exist a positive integer $n$ together with points $x_1, \dots, x_n$ such that
\begin{equation} \label{co1const}
n \leq \frac{2576d (1+\ve)}{\ve^2}
\end{equation}
and 
$$
D_n^{(\mathcal{C},\mathbb{P})}(x_1, \dots, x_n) \leq \ve.
$$
\end{theorem}

\emph{Remark 1:} This result, together with a variant of Koksma's inequality, proves that QMC integration is feasible even in the high-dimensional dimensional setting, and with respect to general measures. Such a general version of discrepancy theory and QMC integration is described in~\cite{ad} and~\cite{ad2}. Note that discrepancies are usually defined with respect to an \emph{uncountable} class of test sets, such as the class $\mathcal{A}$ in definition~\eqref{dn}. However, often it is possible to replace this uncountable class by a countable class without changing the value of the discrepancy (in the case of~\eqref{dn}, one may replace the class $\mathcal{A}$ by the subclass of those boxes whose upper-right corner has only rational coordinates).\\

\emph{Remark 2:} While Theorem~\ref{th3} needs to be translated into a bound for the maximal error in numerical integration by means of a Koksma--Hlawka inequality, Theorem~\ref{th2} \emph{directly} implies the existence of well-suited points for QMC integration for a VC subgraph class $\mathcal{F}$ of functions. However, for a direct application of Theorem~\ref{th2} it would be necessary to calculate the entropy of an ``interesting'' class of functions whose elements one wishes to integrate numerically. As far as I know, not much is known in this respect, but this seems to be a good starting point for further research.\\

\emph{Remark 3:} In~\cite{ad} we obtained Theorem~\ref{th3} with the upper bound in equation~\eqref{co1const} replaced by $2^{26} d \ve^{-2}$. In other words, the result in Theorem~\ref{th3} improves the previously known estimate by a factor of some tens of thousands, and brings the required number of points within reach of mathematical software. However, in the special case of the star-discrepancy as defined in~\eqref{dn}, our new general upper bound is (not surprisingly) weaker than the best bound currently known, which gives $100 d \ve^{-2}$ instead of the upper bound provided by~\eqref{co1const}; see~\cite{ac}. In this context, confer also~\cite{doerr} and ~\cite{hinw}.\\

Throughout the remaining part of this paper, the following setting will be fixed. $(\mathbb{X}, \mathcal{X})$ denotes a measurable space, $\p$ denotes a probability measure on this space, and $\xi_1, \xi_2, \xi_3, \dots$ are independent, identically distributed (i.i.d.) random variables having distribution $\p$. Furthermore, $\mathcal{C}$ denotes a countable class of elements of $\mathcal{X}$, and $\mathcal{F}$ denotes a countable family of real-valued measurable functions on $(\mathbb{X}, \mathcal{X})$. 

\section{Massart's concentration inequality for empirical processes}

The key ingredient in our proof is the following result of Massart~\cite{massart}. The main aim of Massart's paper was to provide an explicit version of an earlier result of Talagrand, which can be seen from the title \emph{About the constants in {T}alagrand's concentration inequalities for empirical processes} of Massart's paper. Massart's result is explicit in the sense that it eliminates all unspecified constants from Talagrand's result; however, it still contains a certain quantity which depends crucially on the entropy of the class of test functions. We state Massart's result below, and continue our discussion subsequently.

\begin{lemma} \label{lemmamass}
Assume that $\|f\|_\infty \leq b < \infty$ for all $f \in \mathcal{F}$. Let $Z$ denote either 
\begin{equation*} \label{zet}
\sup_{f \in \mathcal{F}} \left| \sum_{i=1}^n f(\xi_i) \right| \qquad \textrm{or} \qquad \sup_{f \in \mathcal{F}} \left| \sum_{i=1}^n f(\xi_i) - \mathbb{E} \left( f (\xi_i) \right) \right|.
\end{equation*}
Let $\sigma^2 = \sup_{f \in \mathcal{F}} \mathbb{V}(f(\xi_i))$. Then for all positive real numbers $\eta$ and $x$
$$
\p \left( Z \geq (1 + \eta) \mathbb{E}(Z) + \sigma \sqrt{2 \kappa n x} + \kappa(\eta) b x \right) \leq e^{-x},
$$
where $\kappa$ and $\kappa(\eta)$ can be taken equal to $\kappa=4$ and $\kappa(\eta)=2.5 + 32 \eta^{-1}$. Moreover, one also has 
$$
\p \left( Z \leq (1 - \eta) \mathbb{E}(Z) - \sigma  \sqrt{2 \kappa' n x} - \kappa'(\eta) b x \right) \leq e^{-x},
$$
where $\kappa'=5.4$ and $\kappa'(\eta)=2.5+43.2\eta^{-1}$.
\end{lemma}

In view of Massart's result, the main remaining part in order to obtain our desired results is to find good upper bounds for $\mathbb{E}(Z)$. Such bounds will be given in the next section.

\section{The expected value of the supremum of an empirical process} \label{sec:lemma}

The following lemma will be deduced from Lemma 13.5 of ~\cite{blm}, together with Haussler's entropy bounds for VC classes. Together with Lemma~\ref{lemmamass} it  allows us to obtain a fully explicit large deviation inequality for empirical processes indexed by sets. The case of empirical processes indexed by functions is treated subsequently.

\begin{lemma} \label{lemmablm}
Assume that $\mathcal{C}$ is a VC class with index $d$, and assume that $\p(C) \leq 1/2$ for all $C \in \mathcal{C}$. Set 
$$
Z = \sup_{C \in \mathcal{C}} \left| \sum_{i=1}^n \left( \mathds{1}_{C} (\xi_i) - \mathbb{P} (C) \right) \right|.
$$
Then
$$
\E (Z) \leq 1015 d \left(1+ \sqrt{1+\frac{2 n}{1015 d}} \right).
$$
\end{lemma}

\begin{proof}[Proof of Lemma~\ref{lemmablm}:~] Let
$$
Z^+ = \sup_{C \in \mathcal{C}} \left(\sum_{i=1}^n \left(\mathds{1}_C(\xi_i) - \mathbb{P}(C) \right) \right).
$$
Adopting the notation from~\cite[Lemma 13.5]{blm} we choose $\sigma^2=1/2$ and set
$$
D = 6 \sum_{j=0}^\infty 2^{-j} \sqrt{H\left(2^{-j-3/2},\mathcal{C}\right)}.
$$
Here $H(\delta,\mathcal{C})$ is the \emph{universal $\delta$-metric entropy} (also called \emph{Koltchinskii--Pollard entropy}) of $\mathcal{C}$ which by~\cite[Lemma 13.6]{blm} is bounded by
$$
H(\delta,\mathcal{C}) \leq 2d \log(e/\delta) + \log(e(d+1)), \qquad \delta>0,
$$
since by assumption $\mathcal{C}$ is a VC class with index $d$. (Note that~\cite[Lemma 13.6]{blm} is just a reformulation of Haussler's entropy bounds, that is, of~\cite[Corollary 1]{hauss}). Thus we have
$$
D \leq 6 \sum_{j=0}^\infty 2^{-j} \sqrt{2d \log(e 2^{j+3/2}) + \log(e(d+1))},
$$
and some standard calculations show that 
\begin{equation} \label{D}
D \leq 31.858 \sqrt{d} \qquad \textrm{for} \qquad d=1,2,\dots.
\end{equation}
The last inequality on page 372 of~\cite{blm} states that
$$
\E (Z^+) \leq \frac{D^2}{2} \left(1+ \sqrt{1+\frac{4 \sigma^2 n}{D^2}} \right)
$$
(note that in~\cite{blm} an additional factor $\sqrt{n}$ appears, which comes from the normalization in their definition of $Z^+$). Together with~\eqref{D} and the fact that $(31.858)^2 \leq 1015$ this implies that
\begin{equation} \label{z1b}
\E (Z^+) \leq \frac{1015}{2} d \left(1+ \sqrt{1+\frac{2n}{1015 d}} \right).
\end{equation}
As noted on page 373 of~\cite{blm}, the same bound as in~\eqref{z1b} holds for $\E(Z^-)$, where 
$$
Z^- = \sup_{C \in \mathcal{C}} \left(\sum_{i=1}^n \left(\mathbb{P}(C) - \mathds{1}_C(\xi_i) \right) \right).
$$
Thus overall we have
\begin{eqnarray*}
\E \left(\sup_{C \in \mathcal{C}} \left|\sum_{i=1}^n \left(\mathds{1}_C(\xi_i) - \mathbb{P}(C) \right) \right|\right) & \leq & \E(Z^+) + \E(Z^-) \\
& \leq & 1015 d \left(1+ \sqrt{1+\frac{2 n}{1015 d}} \right).
\end{eqnarray*}
This proves the lemma.
\end{proof}

For future reference we also state the following lemma, which is a variant of Lemma~\ref{lemmablm} without the assumption that $\p(C) \leq 1/2$ for all $C \in \mathcal{C}$. The proof of Lemma~\ref{lemmablm2} can be given in the same way as the proof of Lemma~\ref{lemmablm}.
\begin{lemma} \label{lemmablm2}
If we skip the assumption that $\p(C) \leq 1/2$ for all $C \in \mathcal{C}$ in the statement of Lemma~\ref{lemmablm}, we obtain
$$
\E (Z) \leq 914 d \left(1+ \sqrt{1+\frac{4n}{914 d}} \right).
$$
\end{lemma}

The following lemma provides an upper bound for the quantity $\mathbb{E}(Z)$ in the case of empirical processes indexed by functions, subject to entropy bounds on the class of functions. The appropriate entropy bounds are contained in Haussler's paper and are stated in Lemma~\ref{lemmahau} and, in a form more suitable for our application, in Lemma~\ref{lemmaent}.

\begin{lemma}[{\cite[Proposition 3]{gine}}] \label{lemmagine}
Assume that the functions in $\mathcal{F}$ are $\mathbb{P}$-centered and that their absolute values are uniformly bounded by a constant $U$. Assume also that for all finitely supported probability measures $\mathbb{Q}$, the $L^2(\mathbb{Q})$-covering numbers satisfy 
$$
N \left(\delta, \mathcal{F},L^2(\mathbb{Q}) \right) \leq \left( \frac{A U}{\delta} \right)^v, \qquad 0 < \delta < U,
$$
for some $A > e$ and $v \geq 2$. Let $\sigma \leq U$ be such that 
$$
\sigma^2 \geq \sup_{f \in \mathcal{F}} ~\mathbb{E} (f^2(\xi_i))
$$
Then for all $n \geq 1$ for 
$$
Z = \sup_{f \in \mathcal{F}} \left| \sum_{i=1}^n f(\xi_i) \right| 
$$
we have
$$
\mathbb{E} (Z) \leq 30 \sqrt{2 v} \sqrt{n \sigma^2 \log (5AU/\sigma)} + 15^2 2^5 v U \log (5AU/\sigma).
$$
\end{lemma}

As noted, the required entropy bounds can be deduced from Haussler's results. 
\begin{lemma}[{\cite[Corollary 3]{hauss}}] \label{lemmahau}
Let $\mathbb{Q}$ be a probability distribution on $(\mathbb{X}, \mathcal{X})$. Assume that the functions in $\mathcal{F}$ have only values in $[0,1]$, and that $\mathcal{F}$ is a VC subgraph class with index $d$. Then for every $\delta>0$ the $L^1$-packing number of $\mathcal{F}$ satisfies
$$
\mathcal{M} \left(\delta, \mathcal{F}, L^1(\mathbb{Q})\right) \leq e (d+1) \left( \frac{2e}{\delta} \right)^d.
$$
\end{lemma}

\begin{lemma} \label{lemmaent}
Let $\mathbb{Q}$ be a probability distribution on $(\mathbb{X}, \mathcal{X})$. Assume that the functions in $\mathcal{F}$ are $\mathbb{P}$-centered and that their absolute values are uniformly bounded by 1. Assume also that $\mathcal{F}$ is a VC subgraph class of index $d$. Then for every $\delta>0$ 
$$
N\left(\delta, \mathcal{F},L^2(\mathbb{Q})\right) \leq \left(\frac{11}{\delta}\right)^{2d}.
$$
\end{lemma}

\begin{proof}[Proof of Lemma~\ref{lemmaent}:~] There are three differences between the entropy bound in Lemma~\ref{lemmahau} and the desired bound in Lemma~\ref{lemmaent}. First, the bound in Lemma~\ref{lemmahau} is stated for packing numbers, while we need a bound for covering numbers. Second, the bound in Lemma~\ref{lemmahau} is for the $L^2$-distance, while we need a bound for the $L^1$-distance. Third, the bound in Lemma~\ref{lemmahau} is formulated for a class of functions having values in $[0,1]$, while we need a bound for centered functions (which also have non-negative values).\\

Let $\mathcal{F}$ be the class of functions from the assumption of the lemma, and let $\mathcal{G}$ be the class of all functions of the form $g(x) = (f(x)+1)/2$, where $f \in \mathcal{F}$. Then $\mathcal{G}$ is a class of functions having values in $[0,1]$, and $\mathcal{G}$ also is a VC subgraph class of index $d$. It is quite easily seen that for every measure $\mathbb{Q}$ we have
$$
\mathcal{M} \left(\delta, \mathcal{F}, L^1(\mathbb{Q})\right) \leq \mathcal{M} \left(\delta/2, \mathcal{G}, L^1(\mathbb{Q})\right)
$$
for every $\delta>0$. By a well-known relation between packing numbers and covering numbers (see for example~\cite[p. 98]{vdv}) we have
$$
N \left(\delta,\mathcal{F},L^1(\mathbb{Q})\right) \leq \mathcal{M} \left(\delta/2, \mathcal{F}, L^1(\mathbb{Q})\right)
$$
for every $\delta>0$. Furthermore we have
$$
N \left(\delta,\mathcal{F},L^2(\mathbb{Q})\right) \leq N \left(\delta^2,\mathcal{F},L^1(\mathbb{Q})\right), \qquad \delta>0,
$$
since all functions in $\mathcal{F}$ have values in $[-1,1]$; this follows directly from
$$
\int |f-g|^2 ~d \mathbb{Q} \leq \int |f-g| ~d \mathbb{Q}.
$$
Consequently from Lemma~\ref{lemmahau} we can deduce that overall we have
\begin{eqnarray*}
N\left(\delta,\mathcal{F},L^2(\mathbb{Q})\right) & \leq & \mathcal{M} \left(\delta^2/4, \mathcal{F}, L^1(\mathbb{Q})\right) \\
& \leq & e (d+1) \left( \frac{8e}{\delta^2} \right)^d \\
& \leq & \left(\frac{11}{\delta}\right)^{2d},
\end{eqnarray*}
which is valid for all $\delta \in (0,1)$ and $d \geq 1$.
\end{proof}

From a combination of Lemma~\ref{lemmagine} and Lemma~\ref{lemmaent} we directly obtain the following proposition.

\begin{proposition} \label{prop}
Assume that the functions in $\mathcal{F}$ are $\mathbb{P}$-centered and that their absolute values are uniformly bounded by 1. Assume also that $\mathcal{F}$ is a VC subgraph class of index $d$. Let $\sigma$ be such that 
$$
\sigma^2 \geq \sup_{f \in \mathcal{F}} ~\mathbb{E} (f^2(\xi_i)),
$$
 Then for all $n \geq 1$ for 
$$
Z = \sup_{f \in \mathcal{F}} \left| \sum_{i=1}^n f(\xi_i) \right| 
$$
we have
$$
\mathbb{E} (Z) \leq 60 \sqrt{d} \sqrt{n \sigma^2 \log (55/\sigma)} + 15^2 2^6 d \log (55/\sigma).
$$
If we are willing to choose $\sigma=1$, then we have
$$
\mathbb{E} (Z) \leq 121 \sqrt{d n} + 58000 d.
$$\end{proposition}

\section{Proof of Theorems~\ref{th1},~\ref{th2} and~\ref{th3}.}

\begin{proof}[Proof of Theorem~\ref{th1}:~]
We split the class $\mathcal{C}$ into two subclasses $\mathcal{C}_1$ and $\mathcal{C}_2$ such that 
$$
\mathbb{P} (C) \leq \frac{1}{2} \qquad \textrm{for all $C \in \mathcal{C}_1$, \quad and } \qquad \mathbb{P} (C) > \frac{1}{2} \qquad \textrm{for all $C \in \mathcal{C}_2$.}
$$
Now we define $\mathcal{C}_3$ as the class of all the sets $\mathbb{X} \backslash C,~C \in \mathcal{C}_2$. Then trivially $\mathcal{C}_1$ has VC index at most $d$, and it is easily seen that $\mathcal{C}_3$ also has VC index at most $d$ (note that the VC index does not change when replacing a class of sets by the class of all complements of these sets). Furthermore, we also have
$$
\mathbb{P} (C) \leq \frac{1}{2} \qquad \textrm{for all $C \in \mathcal{C}_3$}.
$$
Consequently we may apply Lemma~\ref{lemmablm} to the classes $\mathcal{C}_1$ and $\mathcal{C}_3$. Note that we have
\begin{eqnarray}
Z & = & \sup_{C \in \mathcal{C}} \left|\sum_{i=1}^n \left(\mathds{1}_C(\xi_i) - \mathbb{P}(C) \right) \right| \nonumber\\
& \leq &  \max \left(\sup_{C \in \mathcal{C}_1} \left|\sum_{i=1}^n \left(\mathds{1}_C(\xi_i) - \mathbb{P}(C) \right) \right|,~ \sup_{C \in \mathcal{C}_3} \left|\sum_{i=1}^n \left(\mathds{1}_C(\xi_i) - \mathbb{P}(C) \right) \right|\right). \label{c1c2*}
\end{eqnarray}
Combining the result from Lemma~\ref{lemmablm} with~\eqref{c1c2*} and Lemma~\ref{lemmamass} we obtain that for all positive real numbers $\eta$ and $x$ we have
\begin{eqnarray*}
\textup{Pr} \left( Z > (1 + \eta) 1015 d \left(1+ \sqrt{1+\frac{2 n}{1015 d}} \right) + \sqrt{n \kappa x} + \kappa(\eta) x \right) \leq 2 e^{-x},
\end{eqnarray*}
%Choosing $\eta=1/32$, $\kappa=4$ and $\kappa(\eta)=2.5 + 32 \eta^{-1}$ this yields
%$$
%\textup{Pr} \left( \sup_{C \in \mathcal{C}} \left|\sum_{i=1}^n \left(\mathds{1}_C(\xi_i) - \mathbb{P}(C) \right) \right| > 33 d \left(1+ \sqrt{1+\frac{n}{512 d}} \right) + \sqrt{2 n x} + 3.5 x \right) \leq 2 e^{-x},
%$$
where $\kappa$ and $\kappa(\eta)$ can be taken equal to $\kappa=4$ and $\kappa(\eta)=2.5 + 32 \eta^{-1}$. This proves Theorem~\ref{th1}.\\
\end{proof}

\begin{proof}[Proof of Theorem~\ref{th2}:~]
Under the assumptions of Theorem~\ref{th2}, by combining Lemma~\ref{lemmamass} and Proposition~\ref{prop} we obtain
\begin{eqnarray*}
& & \textup{Pr} \bigg( Z > (1 + \eta) \left( 60 \sqrt{d} \sqrt{n \sigma^2 \log (55/\sigma)} + 15^2 2^6 d \log (55/\sigma) \right) \qquad \qquad \qquad \qquad \\
& & \qquad \qquad \qquad \qquad \qquad  \qquad + \sigma \sqrt{8 n x} + \left( 2.5 + 32 \eta^{-1} \right) x \bigg) \leq e^{-x},
\end{eqnarray*}
for every $\eta>0$ and $x>0$. If we are willing to choose $\sigma=1$, then we obtain
$$
\textup{Pr} \left( Z > (1 + \eta) \left( 121 \sqrt{d} \sqrt{n} + 58000 d \right) + \sqrt{8 n x} + \left(2.5 + 32 \eta^{-1}\right) x \right) \leq e^{-x},
$$
for every $\eta>0$ and $x>0$.\\
\end{proof}
%Note that in the previous inequality, the factor $121 \sqrt{d} \sqrt{n}$ dominates the factor $58000 d$ whenever $n \geq 230000 d$. 

\begin{proof}[Proof of Theorem~\ref{th3}:~]
Let
$$
n = \left\lceil 2575.5 ~ d (1+\ve) \ve^{-2} \right\rceil,
$$
and apply Theorem~\ref{th1} with the specific choice of $\eta=1/14$ and $x = (\log 2)+10^{-10}$. Then the right-hand side in the conclusion of Theorem~\ref{th1} is smaller than 1, which means that there exists a realization of $x_1,\dots, x_n$ of the random variables $\xi_1, \dots, \xi_n$ such that 
\begin{eqnarray}
D_n^{(\mathcal{C},\mathbb{P})}(x_1, \dots, x_n)  & = & \frac{1}{n} \sup_{C \in \mathcal{C}} \left|\sum_{i=1}^n \left(\mathds{1}_C(x_i) - \mathbb{P}(C) \right) \right| \nonumber\\
& \leq & \frac{1}{n} \left(1015(1+\eta) d \left(1+ \sqrt{1+\frac{2n}{1015 d}} \right) + 2 \sqrt{n x} + \left(2.5+32\eta^{-1}\right) x\right). \label{ourd}
\end{eqnarray}
Now it is a matter of computational routine to check that with our choice of $n, ~\eta$ and $x$ the term in~\eqref{ourd} is actually dominated by $\ve$. Furthermore, one can easily check that 
$$
\left\lceil 2575.5 ~ d (1+\ve) \ve^{-2} \right\rceil \leq 2576 d (1+\ve) \ve^{-2}
$$
for all $\ve \in (0,1)$ and $d=1,2,\dots$, which concludes the proof of the theorem.
\end{proof}

\bibliography{Concentration_inequalities_bib}
\bibliographystyle{abbrv}

\end{document}